\documentclass[12pt]{article}
\usepackage{latexsym}
\usepackage{amsmath,amsfonts,amssymb,mathrsfs,graphicx}
\usepackage{amsthm}
\usepackage{xcolor}
\usepackage{verbatim}

\title{Self-stabilizing processes based on random signs \\
}
\author{K.J. Falconer \\
\small{{\em Mathematical Institute,
University of St~Andrews, North Haugh, St~Andrews,}} \\
\small{{\em Fife, KY16~9SS, Scotland, UK }} \\
\small{ and } \\
J. L\'{e}vy V\'{e}hel \\
\small{{\em Case Law Analytics \& Inria,}} \\
\small{{\em Universit\'{e} Nantes, Laboratoire de Math\'{e}matiques 
Jean Leray}}\\
\small{{\em2 Rue de la Houssini\`{e}re - BP 92208 - F-44322 Nantes Cedex, France}}}

\newcommand{\E}{\mathbb{E}}
\renewcommand{\P}{\mathbb{P}}

\newcommand\ed{\stackrel{{\rm dist}}{=}}

\newcommand\topp{\stackrel{{\rm p}}{\rightarrow}}
\newcommand\tod{\stackrel{{\rm dist}}{\rightarrow}}

\newcommand{\be}{\begin{equation}} 
\newcommand{\ee}{\end{equation}} 

\newcommand\X{{\sf X}}
\newcommand\Y{{\sf Y}}
\newcommand\Pt{{\sf P}}

 \newtheorem{theo}{Theorem}[section]
 
 \newtheorem{lem}[theo]{Lemma}
 \newtheorem{prop}[theo]{Proposition}

\begin{document}
\maketitle

\begin{abstract}
\noindent  A self-stabilizing processes $\{Z(t), t\in [t_0,t_1)\}$ is a random process which when localized, that is scaled to a fine limit near a given $t\in [t_0,t_1)$, has the distribution of an $\alpha(Z(t))$-stable process, where $\alpha: \mathbb{R}\to (0,2)$ is a given continuous function. Thus the stability index near $t$ depends on the value of the process at $t$. In another paper \cite{FL2} we constructed self-stabilizing processes using sums over plane Poisson point processes in the case of $\alpha: \mathbb{R}\to (0,1)$ which depended on the almost sure absolute convergence of the sums. Here we construct pure jump self-stabilizing processes when $\alpha$ may take values greater than 1 when convergence may no longer be absolute. We do this in two stages, firstly by setting up a process based on a fixed point set but taking random signs of the summands, and then randomizing the point set to get a process with the desired local properties.
 \end{abstract}

\section{Introduction and background}
\setcounter{equation}{0}
\setcounter{theo}{0}
For a fixed $ 0<\alpha \leq 2$ {\it symmetric $\alpha$-stable L\'{e}vy motion} $\{L_\alpha (t), t\geq 0\}$   is a stochastic process characterized by having stationary independent increments with $L(0) = 0$ almost surely, and 
$L_\alpha (t) - L_\alpha (s)\ (s>t)$ having the distribution of $S_\alpha((t-s)^{1/\alpha},0,0)$,
where $S_\alpha(c, \beta,\mu)$ denotes a stable random variable with stability-index $\alpha$, with  scale parameter $c$, skewness parameter $\beta$ and shift $\mu$.  A detailed account of such processes may be found in \cite{Bk_Sam} but  we summarize here the  features we need.
Stable motion $L_\alpha $  is $1/\alpha$-self-similar in the sense that $L_\alpha (c t)$ and $c^{1/\alpha}L_\alpha (t)$ are equal in distribution so in particular have the same finite-dimensional distributions. There is a version of  $L_\alpha $ such that its sample paths are c\`{a}dl\`{a}g, that is right continuous with left limits. 

One way of representing  symmetric $\alpha$-stable L\'{e}vy motion $L_\alpha$ is as a sum over a plane point process.
 Throughout the paper we write
$$
r^{\langle s\rangle} = {\rm sign}(r)|r|^s \mbox{ for }r\in \mathbb{R}, s\in \mathbb{R}.
$$
Then
\begin{equation}\label{sum}
L_\alpha  (t) = C_\alpha 
\sum_{(\X,\Y) \in \Pi} 1_{(0,t]}(\X) \Y^{\langle-1/\alpha\rangle},
\end{equation}
where $C_\alpha$ is a normalising constant given by
$$
C_\alpha = \Big(\int_0^\infty u^{-\alpha} \sin u\, du \Big)^{-1/\alpha}
$$
and where $\Pi$  is a Poisson point process on $\mathbb{R}^+ \times\mathbb{R}$  with plane Lebesgue measure $\mathcal{L}^2$ as mean measure, so that for a Borel set $A \subset \mathbb{R}^+ \times\mathbb{R}$ the number of points of $\Pi$ in $A$ has a Poisson distribution with  parameter $\mathcal{L}^2(A)$, independently for disjoint $A$. 
The sum \eqref{sum} is almost surely absolutely uniformly convergent if $0<\alpha<1$, but if $\alpha\geq 1$ then \eqref{sum}  must be taken as the limit as $n\to \infty$ of symmetric partial sums
$$
L_{\alpha ,n} (t) = C_\alpha 
\sum_{(\X,\Y) \in \Pi : |\Y|\leq n} 1_{(0,t]}(\X) \Y^{\langle-1/\alpha\rangle},
$$
in the sense that $\|L_{\alpha ,n}-L_\alpha\|_{\infty}\to 0$ almost surely.

Several variants of $\alpha$-stable motion have been considered. 
For example, for {\it multistable L\'{e}vy motion} $\{M_{\alpha} (t), t\geq 0\}$ the stability index $\alpha$ in  \eqref{sum} can depend on  $\X$ so that the local behaviour changes with $t$, see \cite{FLL,FL,KFLL,XFL,LLA,LLL,LL}.  Thus given a continuous $\alpha: \mathbb{R}^+ \to (0,2)$, 
$$
M_{\alpha}(t) =  
\sum_{(\X,\Y) \in \Pi} 1_{(0,t]}(\X) C_{\alpha(\X)} \Y^{\langle-1/\alpha(\X)\rangle}.
$$
Then $M_{\alpha}$ is a Markov process. Under certain conditions it is {\it localisable}  with {\it local form} $L_{\alpha(t)}$, in the sense that near $t$ the process `looks like' an $\alpha(t)$-stable process, that is for each $t>0$ and $u \in \mathbb{R}$,
$$\frac{M_{\alpha}(t+ru) -M_{\alpha}(t)}{r^{1/\alpha(t)}} \tod L_{\alpha(t)}(u)$$
as $r\searrow 0$, where convergence is in distribution with respect to the Skorohod metric and consequently  is convergent in finite dimensional distributions, see \cite{FLL,FL}. 

The local stability parameter of  multistable L\'{e}vy motion depends on the time $t$ but in some contexts, for example in financial modelling,  it may be  appropriate for the local stability parameter to depend instead (or even as well) on the {\it value} of the process at time $t$. Such a process might be called `self-stabilizing'. Thus, for suitable 
$\alpha: \mathbb{R} \to (0,2)$, we seek a process $\{Z (t), t\geq 0\}$ that is localisable with  local form $L^0_{\alpha(Z (t))}$, in the sense that for each $t$ and $u>0$,
\begin{equation}\label{locdef1}
\frac{Z(t+ru) -Z(t)}{r^{1/\alpha(Z (t))}}\bigg|\,  \mathcal{F}_t \  \tod \ L^0_{\alpha(Z (t))}(u)\end{equation}
as $r\searrow 0$,
where convergence is in distribution and finite dimensional distributions and where $\mathcal{F}_t$ indicates conditioning on the process up to time $t$. (For notational simplicity it is easier to construct  $Z_{\alpha}$ with the non-normalised $\alpha$-stable processes $L^0_{\alpha} = C_\alpha^{-1} L_{\alpha}$ as its local form.)

Throughout the paper we write $D[t_0,t_1)$ for the c\`{a}dl\`{a}g functions on the interval $[t_0,t_1)$, that is functions that are right continuous with left limits; this is the natural space for functions defined by sums over point sets. 

In an earlier paper \cite{FL2} we constructed self-stabilizing processes for  $\alpha: \mathbb{R}^+ \to (0,1)$  by first showing that there exists a deterministic function $f\in D[t_0,t_1)$ satisfying the relation
$$f(t) = a_0 +\sum_{(x,y) \in \Pi} 1_{(t_0,t]}(x) y^{\langle-1/\alpha(f(x_-))\rangle}$$
for a {\it fixed} point set $\Pi$, and then randomising to get a random function $Z$ such that
\begin{equation*}
Z  (t) =  a_0 + \sum_{(\X,\Y) \in \Pi} 1_{(t_0,t]}(\X)\, \Y^{\langle-1/\alpha(Z  (\X_-))\rangle} \qquad (t_0\leq t <t_1).
\end{equation*}
Then, for all $t\in [t_0,t_1)$ this random function satisfies \eqref{locdef1} almost surely. However, this approach depends on the infinite sums being absolutely convergent, which need not be the case if $\alpha(t) \geq1$ for some $t$.

Here we use an alternative approach to construct self-stabilizing processes where $\alpha: \mathbb{R}^+ \to (0,2)$ and in general we cannot assume absolute convergence of the sums. We  show in Section \ref{sec:signs}  that for a fixed point set $\Pi^+\subset (t_0,t_1)\times \mathbb{R}^+$ and independent random `signs' $S(x,y)=\pm1$ there exists almost surely a random function $Z\in D[t_0,t_1)$ satisfying
$$ Z(t) = a_0 +\sum_{(x,y) \in \Pi^+} 1_{(t_0,t]}(x) S(x,y)y^{-1/\alpha (Z(x_-))}\qquad (t_0 \leq t <t_1), $$
see Theorem \ref{finlim2}. To achieve this we work with partial sums 
$$ Z_n(t) = a_0 +\sum_{(x,y) \in \Pi^+:|y|\leq n} 1_{(t_0,t]}(x) S(x,y)y^{-1/\alpha (Z(x_-))}\qquad (t_0 \leq t <t_1) $$
and show that the limit as $n\to \infty$ exists in a norm given by $\E\big(\|\cdot\|_\infty^2\big)^{1/2}$, where $\E$ denotes expectation. This is more awkward than it might seem at first sight since,  as $n$ increases, if a new  point $(x,y)\in \Pi^+$  enters the sum, then $Z_n(t)$ will change for $t\geq x$ so for all  $(x',y')$ with $x'>x$ and  $y'<y$    the summands ${y'}^{\langle-1/\alpha(Z_n({x'}_{-}))\rangle}$ will change, with a knock on effect so that the change in $Z_n(t)$ may be considerably amplified as $t$ increases past further $x$ with $(x,y)\in \Pi^+$ and $y< n$.

In Section \ref{sec:rand} we randomise the construction further by taking $\Pi^+$ to be a Poisson point process on $ (t_0,t_1)\times \mathbb{R}^+$ with  mean measure $2{\mathcal L}^2$ which, combined with the random signs, gives a point process with the same distribution as $\Pi$ on $ (t_0,t_1)\times \mathbb{R}$.  We show that the resulting process $Z$ satisfies a H\"{o}lder continuity property and is self-stabilizing in the sense of \eqref{locdef1}, see Theorem \ref{rtloc}.

\subsection{Basic facts used throughout the paper}\label{sec1.1}
For the rest of the paper we fix $a_0 \in \mathbb{R}$ and  $0<a <b<2 $ together with a function $\alpha :\mathbb{R} \to [a,b]$ that is continuously differentiable with bounded derivative.
By the mean value theorem,
$$ y^{-1/\alpha(v)} - y^{-1/\alpha(u)}
\ =\ (v- u)
y^{-1/\alpha(\xi)}\log y \frac{\alpha'(\xi)}{\alpha(\xi)^2}\qquad (y>0,\ u,v \in \mathbb{R} ),$$
where $\xi \in (u,v)$. In particular this gives the  estimate we will use frequently:
\be 
\big| y^{-1/\alpha(v)} - y^{-1/\alpha(u)}\big |
\  \leq  \ M\  |v- u|\, y^{-1/(a,b)}\qquad (y>0, \ u,v \in \mathbb{R}), \label{ydif4}
 \ee
where
$$
M\  =\  \sup_{\xi\in\mathbb{R}} \frac{|\alpha'(\xi)|}{\alpha(\xi)^2},
$$
and for convenience we write
$$
y^{-1/(a,b)} \ = \ \max\big\{ y^{-1/a}\big(1+|\log y| \big), y^{-1/b}\big(1+|\log y| \big)\big\}
\qquad (y>0)
$$
and 
$$ y^{-2/(a,b)} \ =\ (y^{-1/(a,b)})^2.$$ 

For $t_0< t_1$ and a suitable probability space $\Omega$ (to be specified later), we will work with  functions 
$F: \Omega\times [t_0, t_1) \to \mathbb{R}\cup \{\infty \}$ which we assume to be measurable (taking  Lebesgue measure on  $[t_0, t_1)$).  Writing  $F_\omega(t)$ for the value of $F$ at $\omega \in \Omega$ and  $t\in [t_0, t_1)$, we think of   $F_\omega$ as a random function on  $[t_0, t_1)$ in the natural way (most of the time we will write $F$ instead of $F_\omega$ when the underlying randomness is clear). In particular we will work in the space 
$${\mathcal D} \ = \ \big\{F: F_\omega \in D[t_0, t_1) \mbox{ for almost all }
 \omega \in \Omega \mbox{ with } \E\big(\|F\|_\infty^2\big)<\infty\big\},$$
 where $\E$ is expectation, $D[t_0,t_1)$ denotes the c\`{a}dl\`{a}g functions, and $\|\cdot\|_\infty$ is the usual supremum norm. By identifying $F$ and $F'$ if $F_\omega = F'_\omega$ for almost all $\omega \in \Omega$, this becomes a normed space under the norm
\be\label{norm}
\Big(\E\big(\|F\|_\infty^2\big)\Big)^{1/2}.
\ee
A routine check shows that \eqref{norm} defines a complete norm on ${\mathcal D}$.

\section{Point sums with random signs}\label{sec:signs}
\setcounter{equation}{0}
\setcounter{theo}{0}

In this section we fix a discrete point set $\Pi^+ \subset (t_0,t_1)\times \mathbb{R^+}$ 
and form sums over values at the points of $\Pi^+$ with an independent random assignment of sign $+$ or $-$ at each point of $\Pi^+$.

We will assume that the point set $\Pi^+$ satisfies 
$$
 \sum_{(x,y) \in \Pi^+}y^{-2/(a,b)}< \infty;
$$
this will certainly be the case if $ \sum_{(x,y) \in \Pi^+}y^{-2/b'}< \infty$ for some $b'$ with  $b<b'<2$.

Our first aim is to show that if $\{S(x,y) \in \{-1,1\}: (x,y) \in \Pi^+\}$ are random `signs', that is independent random variables taking the values $1$ and $-1$ with equal probability $\frac{1}{2}$, then, almost surely, there exists  a random function $Z \in D[t_0, t_1)$ satisfying
\begin{equation}\label{identsigns}
Z(t) = a_0 +\sum_{(x,y) \in \Pi^+} 1_{(t_0,t]}(x) S(x,y)y^{-1/\alpha (Z(x_-))}\qquad (t_0 \leq t <t_1). 
\end{equation}
in an appropriate sense. This $Z$ will be the  limit in norm of the random functions obtained by restricting the sums to $y \leq n$. Thus we define for $n\in \mathbb{N}$
\begin{equation}\label{identsignsfin}
Z_n(t) = a_0 +\sum_{(x,y) \in \Pi^+:\, y\leq n} 1_{(t_0,t]}(x) S(x,y)y^{-1/\alpha (Z_n(x_-))}\qquad (t_0 \leq t <t_1).
\end{equation}
Here, and throughout this section, our probability space has  $\Omega$ as the set $ \{-1,1\}^{ \Pi^+}\equiv \{S(x,y) \in \{-1,1\}: (x,y) \in \Pi^+\}$ of all assignments of signs $\pm1$ to the points of $ \Pi^+$, and  $\sigma$-field  generated by the subsets of the form $S \times \{-1,1\}^{ \Pi^+\setminus X}$ for each finite $X\subset  \Pi^+$ and each $S\subset \{-1,1\}^X$. In particular, the probability that a given set of $k$ points of $ \Pi^+$ having any particular assignment of signs is $2^{-k}$.

Note that the sum in \eqref{identsignsfin} is over the finite set $\{(x,y)\in \Pi^+: y\leq n\}$ so, given the $S(x,y)$, the piecewise constant $Z_n(t)\in D[t_0,t_1)$ can be evaluated inductively over increasing $x$ with $(x,y)\in \{\Pi^+: y\leq n\}$  in a finite number of steps. Nevertheless, as has been remarked, we need to be careful about taking limits of $\{Z_n\}_n$ as $n \to \infty$ since on increasing $n$ the contributions from the summands with $0<y<n$  change as the values of $Z_n(x_-)$ change.   Note that  if $0<a<b<1$ the sum  \eqref{identsigns} is absolutely convergent, and this case is considered in \cite{FL2}, but if $0<a<b<2$ then $\alpha(z)$ may take values greater than 1 and care is needed in defining $Z$.  We will show that there exists $Z\in {\mathcal D}$ such that $\E\big(\|Z_n-Z\|_\infty^2\big) \to 0$  and $Z$ satisfies \eqref{identsigns} in an appropriate sense. In particular  there is a sequence of integers $n_j\nearrow\infty$ such that almost surely $ \|Z_{n_j}- Z\|_\infty \to 0$, i.e. $Z_{n_j}$ converges uniformly to $Z$.

\subsection{Existence of functions defined by random signs}\label{sec2.1}

 To obtain such a $Z$ we show that $\{Z_n\}_n$ is a Cauchy sequence in the complete norm \eqref{norm}.
$\E\big(\|\cdot\|_\infty^2\big)^{1/2}$ on ${\mathcal D}$.

It is convenient to make a further assumption on $\Pi^+$, that if $(x,y), (x',y') \in\Pi^+$ then $x\neq x'$. Without this assumption the results remain valid with an essentially identical proof, but the notation becomes more cumbersome as the single terms added in \eqref{mart} have to be replaced by sums over several terms corresponding to each point $(x_i,y_i)$ with a common value of $x_i$. In any case, when in Section 3 we let $\Pi^+$ be a realisation of a Poisson point process, this assumption will hold almost surely.

\begin{prop}\label{finlim1A}
Let  $a_0, \alpha $ and $\Pi^+$ be as above  and  let  $Z_n$ be given by  \eqref{identsignsfin}.  Then for $m\geq n\geq 1$,
\begin{equation}
\E\big(\|Z_m -Z_n\|_\infty^2\big) 
\  \leq  \ 4 \prod_{(x,y)\in \Pi^+ : 0<y\leq n} (1+M^2y^{-2/(a,b)}) \sum_{(x,y)\in \Pi ^+ : n<y\leq m} y^{-2/b}. \label{cauchy0}
 \end{equation}
In particular $\{Z_n\}_n$ is a Cauchy sequence in ${\mathcal D}$ under the norm $\E(\|\cdot\|^2_\infty)^{1/2}$.
\end{prop}

\begin{proof}
Let $m>n$. We list the points
$$\{(x,y) \in \Pi^+ : y\leq m\} = \{(x_1,y_1), \ldots, (x_N,y_N)\},$$
with $t_0<x_1 < \cdots < x_N <t_1$, where as mentioned we assume  that the $x_i$ are distinct. 
For notational convenience we set $x_0 := t_0$ and $x_{N+1} := t_1$. 
We write $i_1<i_2< \cdots<i_K$ for the indices such that $y_{i_k} \leq n$, and let $i_0 := 0$ and $i_{K+1} := N+1$
With this notation,  \eqref{identsignsfin} restricts to the jump points $x_i$ as
\begin{equation}\label{finalt}
Z_m(x_i) = a_0 +\sum_{0<j\leq i} S(x_j,y_j)y_j^{-1/\alpha (Z_m(x_{j-1}))},\quad 
Z_n(x_i) = a_0 +\sum_{k: 0<i_k\leq i} S(x_{i_k},y_{i_k})y_{i_k}^{-1/\alpha (Z_n(x_{{i_k}-1}))}. 
\end{equation}

Write
\begin{equation}\label{cki}
c_k = M^2\, y_{i_k}^{-2/(a,b)} \qquad (1\leq k \leq K)
\end{equation}
and
\begin{equation}\label{epki}
\epsilon_k = \sum_{ i_{k-1}+1\leq i \leq i_k -1}  y_i^{-2/b} \qquad (1\leq k \leq K+1).
\end{equation}
Let $\mathcal{F}_i$ be the minimal $\sigma$-field of subsets of $\Omega$ such that the sign assignments $\{S(x_j,y_j):1\leq j\leq i\}$ are Borel measurable; thus conditioning on $\mathcal{F}_i$ is equivalent to taking the finite set of values $\{S(x_j,y_j): 1\leq j\leq i\}$ as known.
We consider $\big\{Z_n(x_i) - Z_m(x_i),\mathcal{F}_i\big\}_{i=0}^N$ which from \eqref{finalt} and that $\E(S(x_j,y_j)) = 0$ is a bounded martingale. Indeed, from \eqref{finalt}, for $1\leq i\leq N$,
 \begin{align}\label{mart}
Z_m(x_{i}) - & Z_n(x_{i})\ =\   Z_m(x_{i-1}) - Z_n(x_{i-1}) \\
&+ \left\{
\begin{array}{ll}
 S(x_{i},y_{i})\,y_{i}^{-1/\alpha (Z_m(x_{i-1}))}  &  \mbox{ if } i\neq i_k   \mbox{ for all } k   \\
S(x_{i},y_{i})\big(y_{i}^{-1/\alpha (Z_m(x_{i-1}))} -   y_{i}^{-1/\alpha (Z_n(x_{i-1}))}\big) & \mbox{ if } i= i_k      \mbox{ for some } k  \nonumber 
\end{array}.
\right.
\end{align}

We will
show by induction on $i$ that for all $1\leq k\leq K+1$ and all $i_{k-1} \leq i <i_k$,
\begin{eqnarray}
\E\big((Z_m(x_i) - Z_n(x_i))^2\big)
&\leq& (1+c_1)\cdots(1+c_{k-1}) ( \epsilon_1+ \epsilon_2+\cdots + \epsilon_{k-1})\nonumber \\
&&  +\,  y_{i_{k-1}+1}^{-2/b}  + y_{i_{k-1}+2}^{-2/b} + \cdots +  y_{i}^{-2/b}\label{ind1}\\
&\leq& (1+c_1)\cdots(1+c_{k-1}) ( \epsilon_1+ \epsilon_2+\cdots + \epsilon_{k-1})
 \,+\,\epsilon_k.\label{ind2}
\end{eqnarray}
Note that \eqref{ind2} follows immediately from \eqref{ind1} using \eqref{epki}. Inequality  \eqref{ind1} is trivially  true when $i=0$. Let $ 0< i \leq N$ and assume inductively that \eqref{ind1} holds with $i$ replaced by $i-1$.  There are two cases.

(a) If $i\neq i_k$  for all $k $ then from \eqref{mart}
 \begin{align*}
\E\big((Z_m (x_i) - & Z_n(x_i))^2\big| \mathcal{F}_{i-1} \big)\\
&=\ \E\Big(\Big(Z_m(x_{i-1}) - Z_n(x_{i-1}) +  S(x_{i},y_{i})\,y_{i}^{-1/\alpha (Z_m(x_{i-1}))}\Big)^2\ \Big| \mathcal{F}_{i-1}\Big)\\
&=\ \big(Z_m(x_{i-1}) - Z_n(x_{i-1})\big)^2  +  y_{i}^{-2/\alpha (Z_m(x_{i-1}))}\\
&\leq \ \big(Z_m(x_{i-1}) - Z_n(x_{i-1})\big)^2  +  y_{i}^{-2/b}
\end{align*}
as $y_i\neq y_{i_k}$ for all $k$  implies  $y_i>n\geq 1$. Thus  taking the unconditional expectation  and using \eqref{ind1} for $i-1$ gives \eqref{ind1} for $i$.

(b) If $i=  i_k $ for some $1\leq k\leq K$, then from \eqref{mart},
 \begin{align*}
\E\big((Z_m & (x_i)- Z_n(x_i))^2\big| \mathcal{F}_{i-1} \big)\\
&=\ \E\Big(\Big(Z_m(x_{i-1}) - Z_n(x_{i-1}) +  S(x_{i},y_{i})\big(y_{i}^{-1/\alpha (Z_m(x_{i-1}))} -   y_{i}^{-1/\alpha (Z_n(x_{i-1}))}\big)\Big)^2\Big| \mathcal{F}_{i-1}\Big)\\
&=\ \big(Z_m(x_{i-1}) - Z_n(x_{i-1})\big)^2  + \big(y_{i}^{-1/\alpha (Z_m(x_{i-1}))} -   y_{i}^{-1/\alpha (Z_n(x_{i-1}))}\big)^2\\
&\leq \ \big(Z_m(x_{i-1}) - Z_n(x_{i-1})\big)^2\big(1 +  M^2 y_{i_k}^{-2/(a,b)}\big)\\
&\leq \ \big(Z_m(x_{i-1}) - Z_n(x_{i-1})\big)^2\big(1 +  c_k\big),
\end{align*}
 using \eqref{ydif4} and \eqref{cki}.
Again, taking the unconditional expectation and using \eqref{ind2} for $i-1$ gives \eqref{ind1} for $i$ with a vacuous sum of terms $y_j^{-2/b}$, completing the induction.

It follows from  \eqref{ind2} that for all $0\leq i \leq N$,
\begin{eqnarray}
\E\big((Z_m(x_i) - Z_n(x_i))^2\big)& \leq &  \prod_{k=1}^K (1+c_k) \sum_{k= 1}^{K+1} \epsilon_k\nonumber\\ 
 &  \leq  &  \prod_{(x,y)\in \Pi ^+ : y\leq n} (1+M^2y^{-2/(a,b)}) \sum_{(x,y)\in \Pi ^+ : n<y\leq m} y^{-2/b} \label{cauchy4}
\end{eqnarray}
using \eqref{cki} and \eqref{epki}.

Noting that $Z_m(t) -Z_n(t)$ is constant except at the jump points $x_i$, and applying Doob's maximal inequality  \cite{Wil} to the martingale 
$\big\{Z_m(x_i) - Z_n(x_i),\mathcal{F}_i\big\}_{i=0}^N$,
  we obtain
\begin{eqnarray*}
\E\big(\sup_{t_0\leq t < t_1} (Z_m(t) - Z_n(t))^2\big)
&=& \E\big(\max_{0\leq i \leq N} (Z_n(x_i) - Z_m(x_i))^2\big)\\ 
&\leq& 4\, \E \big((Z_m(x_N) - Z_n(x_N))^2\big) 
 \end{eqnarray*}
 for all $m\geq n$. 
 Combining with  \eqref{cauchy4} gives \eqref{cauchy0}.
Since 
$$\prod_{(x,y)\in \Pi ^+ } (1+M^2y^{-2/(a,b)})\ \leq\ \exp\Big(M^2\sum_{(x,y) \in \Pi ^+ } y^{-2/(a,b)}\Big)\ < \infty$$
and $\sum_{(x,y) \in \Pi ^+ } y^{-2/(a,b)}$ and $\sum_{(x,y)\in \Pi ^+ } y^{-2/b}$ are convergent by assumption, $\{Z_n\}_n$ is a Cauchy sequence.
\end{proof}

We now deduce the existence of $Z$ as the norm limit of the $Z_n$. 
 
\begin{theo}\label{finlim2}
Let  $a_0, \alpha $ and $\Pi^+$ be as above. Then there exists $Z \in {\mathcal D}$ satisfying  
\eqref{identsigns},   
 with $\lim_{n\to \infty}\E\big(\|Z_n -Z\|_\infty^2\big)=0$ where $Z_n$ as in  \eqref{identsignsfin};  more specifically
\be \label{diffbound}
\E\big(\|Z_n -Z\|_\infty^2\big)\ \leq\  4\prod_{(x,y)\in \Pi ^+ } (1+M^2y^{-2/(a,b)})\sum_{(x,y)\in \Pi ^+ : y> n} y^{-2/b}\ \to 0. 
\ee
Moreover, there exists a sequence $n_j \nearrow \infty$ such that almost surely $\|Z_{n_j} -Z\|_\infty \to 0$ i.e. $Z_{n_j} \to Z$ uniformly.  If $0<b<1$ then almost surely $\|Z_{n} -Z\|_\infty \to 0$. 
\end{theo}

\begin{proof}
Since $\{Z_n\}_n$ is Cauchy in the complete norm $\E\big(\|\cdot\|_\infty^2\big)^{1/2}$ the desired limit $Z \in {\mathcal D}$ exists. Letting $m\to \infty$ in \eqref{cauchy0} gives \eqref{diffbound}. 

Moreover,  choosing  any increasing sequence $\{{n_j}\}_{j}$ such that
\be\label{rapcon}
 \sum_{(x,y)\in \Pi ^+ : y>n_j} y^{-2/b}\ < \ 2^{-j}
\ee
for all sufficiently large $j$,
then by \eqref{cauchy0} $\E\big(\|Z_{n_{j+1}}-Z_{n_j}\|_\infty^2\big)< c\ 2^{-j}$ so almost surely
 $Z= Z_{n_1} + \sum_{j=1}^\infty (Z_{n_{j+1}}-Z_{n_j})$   is convergent in $\|\cdot\|_\infty$. If $0<b<1$ then
\begin{align*}
\E\big(\|Z_{n+1}&  -Z_n\|_\infty\big)\ \leq\ \big(\E\big(\|Z_{n+1} -Z_n\|_\infty^2\big)\big)^{1/2} \\ 
& \leq \ c \Big(\sum_{(x,y)\in \Pi^+ : n<y\leq n+1} y^{-2/b}\Big)^{1/2} 
 \leq \ c \sum_{(x,y)\in \Pi^+ : n<y\leq n+1} y^{-1/b},
 \end{align*}
so $\E\big(\sum_{n=1}^\infty \|Z_{n+1}  -Z_n\|_\infty\big)<\infty$, giving that
$Z= Z_{1} + \sum_{j=1}^\infty (Z_{n+1}-Z_n) $  is  almost surely uniformly convergent.
 \end{proof}

\subsection{Local properties of functions defined by random signs}

We next examine local properties  of the random function $Z$ constructed  in  Section \ref{sec2.1}. We will show that for a given $t\in [t_0,t_1)$, almost surely $Z$ satisfies a H\"{o}lder condition to the right of $t$   and also is locally approximable by a random function $L$ defined in a similar way to $Z$, but with a fixed exponent $\alpha(Z(t))$.
Throughout this section we fix $t\in [t_0,t_1)$ and throughout this subsection we restrict $\Omega$ to the subspace of full probability (by Theorem \ref{finlim2}) such that there exists a sequence $n_j\nearrow \infty$ such that $Z_{n_j}(t) \to Z(t)$.
Let  $\mathcal{F}_{t}$ be the $\sigma$-field underlying the signs $\{S(x,y)\in \{-1,1\}: (x,y) \in \Pi^+, x\leq t\}$.

Let $Z$ be given by the norm limit of the partial sums $Z_n$ in \eqref{identsignsfin},  as in Section \ref{sec2.1}. 
We also construct $L_n, L \in {\mathcal D}$ restricted to $[t,t_1)$ in a similar way using the  point set $\{(x,y) \in \Pi^+:  (t<x<t_1)\}$ and a (conditional) fixed index  $\alpha(Z(t)) \in [a,b]$; thus  $L_n$ is the piecewise constant random function defined by
$$
L_n (u)\  \equiv\ \sum_{(x,y) \in \Pi^+\,:\, y\leq n} 1_{(t ,u]}(x) S(x,y) y^{-1/\alpha(Z_n(t))}
\qquad (t\leq u <t_1)
$$
where the $S(x,y)$ are independent random signs, and with $L\in {\mathcal D}$ defined by 
\be\label{limL}
\E\big(\|L_n - L\|_\infty^2\big)\to 0,
\ee
 as a particular case of Theorem \ref{finlim2}, where here  $\|\cdot \|_\infty$ is the supremum norm on $[t,t_1)$. 

\begin{prop}\label{prop4.3}
Let $\alpha: \mathbb{R} \to [a,b]$, $\Pi^+$ be as before and let  $Z_n, Z, L_n,L \in {\mathcal D}$ restricted to $[t,t_1)$ be  as above, taking the same realisations of $S(x,y)$ for $Z_n$ and $L_n$.
Then, there are constants $c_1, c_2$ depending only on $\Pi^+$ and $\alpha$ such that, conditional on ${\mathcal F}_{t}$, for all $0\leq h <t_1-t$,
\be\label{espest2}
\E\Big(\sup_{0\leq h'\leq h}\big(Z(t+h')-Z(t)\big)^2\Big) \ \leq \ 
c_1\!\!\!\!\sum_{(x,y)\in \Pi^+\, :\,  t< x\leq t+h}  y^{-2/\alpha(Z(t))} 
\ee
and
\begin{align}
\E\Big(\sup_{0\leq h'\leq h}\big(&(Z(t+h')-Z(t)) -L(t+h')\big)^2\Big)\nonumber\\
& \leq \ 
c_2\Big(\sum_{(x,y)\in \Pi^+\, :\,  t< x\leq t+h}  y^{-2/(a,b)}\Big)
\Big(\sum_{(x,y)\in \Pi^+\, :\,  t< x\leq t+h}  y^{-2/\alpha(Z(t))}\Big).\label{espest2L}
\end{align}
\end{prop}

\begin{proof} 
For brevity write 
$$W_n(u) = Z_n(u)-Z_n(t)\ \ \mbox{ and } \ \ D_n(u) = W_n(u)-L_n(u)    \qquad (t\leq u <t_1).$$
Let $0 \leq h \leq t_1-t$.  
For each $n$, order the points $(x,y) \in \{\Pi^+: t<x\leq t+h, y\leq n\}$ as $(x_i,y_i)$ with  $x_1< x_2 < \cdots < x_K$ and  let $x_0=t$ (as before we lose little other than awkward notation by assuming that the $x_i$ are distinct). Let $\mathcal{F}_i$ be the $\sigma$-field underlying the signs $\{S(x_j,y_j): 1\leq  j\leq i\}$.
Then for $1\leq i \leq K$,
$$W_n(x_i)\  =\  W_n(x_{i-1})   +  S(x_i,y_i)y_i^{-1/\alpha (Z_n(x_{i-1}))}$$
and 
$$
D_n(x_i)\  =\  D_n(x_{i-1})  +  S(x_i,y_i)
\big(y_i^{-1/\alpha (Z_n(x_{i-1}))}- y_i^{-1/\alpha (Z_n(t))}\big)
$$
so  that $\big\{ W_n(x_i), \mathcal{F}_{i}\big\}$ and  $\big\{ D_n(x_i), \mathcal{F}_{i}\big\}$ are bounded martingales.
For each $1\leq i \leq K$, conditioning on $\mathcal{F}_{i-1}$ gives
\begin{eqnarray}
\E\big( W_n(x_i)^2\big| \mathcal{F}_{i-1}\big) 
&=& W_n(x_{i-1})^2 \ +\  y_i^{-2/\alpha (Z_n(x_{i-1}))}\nonumber\\
&\leq& W_n(x_{i-1})^2 \ +\   \big(y_i^{-1/\alpha (Z_n(x_{i-1}))}- y_i^{-1/\alpha (Z_n(t))}+ y_i^{-1/\alpha (Z_n(t))}\big)^2\nonumber\\
&\leq& W_n(x_{i-1})^2 \ +\  2  \big(y_i^{-1/\alpha (Z_n(x_{i-1}))}- y_i^{-1/\alpha (Z_n(t))}\big)^2 
\ +\  2  y_i^{-2/\alpha (Z_n(t))}, \nonumber\\
&\leq& W_n(x_{i-1})^2 \ +\  2 M^2 y_i^{-2/(a,b)}\big(Z_n(x_{i-1}) - Z_n(t)\big)^2 
\ +\  2  y_i^{-2/\alpha (Z_n(t))}\nonumber\\
&\leq& W_n(x_{i-1})^2 \big(1 +\  2 M^2 y_i^{-2/(a,b)}\big) 
\ +\  2  y_i^{-2/\alpha (Z_n(t))}\label{step1} 
\end{eqnarray}
and
\begin{eqnarray}
\E\big( D_n(x_i)^2\big| \mathcal{F}_{i-1}\big) 
&=& D_n(x_{i-1})^2 \ +\big(y_i^{-1/\alpha (Z_n(x_{i-1}))}- y_i^{-1/\alpha (Z_n(t))}\big)^2\nonumber\\
&\leq& D_n(x_{i-1})^2  + M^2 y_i^{-2/(a,b)} (Z_n(x_{i-1})-Z_n(t))^2\nonumber\\
&=& D_n(x_{i-1})^2  + M^2 y_i^{-2/(a,b)} W_n(x_{i-1})^2\label{step2} 
\end{eqnarray}
where we have used \eqref{ydif4}.
Then induction in decreasing $j$ using \eqref{step1} gives
$$
\E\big( W_n(x_K)^2\big| \mathcal{F}_{j}\big) 
\ \leq \
 \prod_{k=j+1}^{K} \big(1+2M^2 y_k^{-2/(a,b)}\big) 
 \Big(W_n(x_{j})^2+2\sum_{k=j+1}^K  y_k^{-2/\alpha(Z_n(t))} \Big)
$$
for  all $0\leq j\leq K-1$, and  induction in decreasing $j$ using \eqref{step1} and \eqref{step2} gives
\begin{eqnarray*}
\E\big( D_n(x_K)^2\big| \mathcal{F}_{j}\big) 
&\leq&\big(M^2 \sum_{k=j+1}^K  y_k^{-2/(a,b)} \big)
 \prod_{k=j+1}^{K-1} \big(1+2M^2 y_k^{-2/(a,b)}\big) \\
 &&+\  \Big(W_n(x_{j})^2+2\sum_{k=j+1}^{K-1} y_k^{-2/\alpha(Z_n(t))} \Big)
+ D_n(x_j)^2.
\end{eqnarray*}
for all $0\leq j\leq K-1$. Setting $j=0$ in these two estimates and noting that $W_n(x_0) = W_n(t)=0$ and $D_n(x_0) = D_n(t)=0$, we get  expectations conditioned only on $ \mathcal{F}_{t}$:
$$
\E\big( W_n(x_K)^2\big) 
\ \leq \
2 \prod_{k=1}^{K} \big(1+2M^2 y_k^{-2/(a,b)}\big) 
 \Big(\sum_{k=1}^K  y_k^{-2/\alpha(Z_n(t))} \Big)$$
and 
$$
\E\big( D_n(x_K)^2\big) 
\ \leq\ 2 M^2 \big(\sum_{k=1}^K  y_k^{-2/(a,b)} \big)
 \prod_{k=1}^{K-1} \big(1+2M^2 y_k^{-2/(a,b)}\big)
  \Big(\sum_{k=1}^{K-1} y_k^{-2/\alpha(Z_n(t))} \Big).
$$
Noting that $W_n$ and $D_n$ are constant between the $x_i$ and applying  Doob's inequality to the martingales $W_n(x_i)$ and $D_n(x_i)$,
\begin{align*}
\E\big(\sup_{0\leq h'\leq h} (& Z_n(t+h)-  Z_n(t))^2\big)\ = \ 
\E\big(\sup_{0\leq h'\leq h} W_n(t+h')^2\big)
\ \leq\  \E\big(\max_{0\leq k \leq K} W_n(x_k)^2\big) \\
&\leq \ 4\E\big( W(x_K)^2\big) \ \leq \ 8c_3  \sum_{k=1}^K y_k^{-2/\alpha(Z_n(t))}
\ \leq \ 8c_3  \sum_{(x,y) \in \Pi^+: t<x\leq t+h}y^{-2/\alpha(Z_n(t))},
\end{align*}
where $c_3 = \prod_{k=1}^{K-1} \big(1+2M^2 y_k^{-2/(a,b)}\big)$, and 
\begin{align*}
\E\Big(\sup_{0\leq h'\leq h} \big(&(Z_n(t+h')-Z_n(t))-L_n(t+h')\big)^2\Big)\ = \ 
\E\big(\sup_{0\leq h'\leq h} D_n(t+h')^2\big)\\
& \leq\  \E\big(\max_{0\leq k \leq K} D_n(x_k)^2\big) \ \leq \ 4\E\big( D(x_K)^2\big) \\
& \leq \ 8M^2 c_3  \big(\sum_{k=1}^K  y_k^{-2/(a,b)} \big)\big(\sum_{k=1}^K y_k^{-2/\alpha(Z_n(t))}\big)\\
& \leq \ 8M^2 c_3  \big(\sum_{(x,y) \in \Pi^+: t<x\leq t+h} y_k^{-2/(a,b)} \big)\big(\sum_{(x,y) \in \Pi^+: t<x\leq t+h} y_k^{-2/\alpha(Z_n(t))}\big).
\end{align*}
By Theorem \ref{finlim2}, $\E\big(\|Z_n -Z\|_\infty^2\big) \to 0$ and $\E\big(\|L_n -L\|_\infty^2\big) \to 0$ and there is a sequence $n_j\nearrow \infty$ such that $Z_{n_j}(t) \to Z(t)$,  so we can take the limit of these inequalities along this subsequence using dominated convergence to get \eqref{espest2} and  \eqref{espest2L}  with $c_1 = 8c_3$ and $c_2=8M^2 c_3$. 
\end{proof}

We remark that versions of \eqref{espest2} and  \eqref{espest2L} with
$$
4\!\!\!\!\sum_{(x,y)\in \Pi^+\, :\,  t< x\leq t+h}  y^{-2/(a,b)} \quad \mbox{ and }\quad
4M^2\Big(\sum_{(x,y)\in \Pi^+\, :\,  t< x\leq t+h}  y^{-2/(a,b)}\Big)^2 .
$$
respectively as the right-hand side bounds  can be obtained using a simpler induction, but the  exponents are not so sharp.

We can immediately deduce a local right H\"{o}lder bound for $Z$ at $t$ as well getting a comparison with $L$. 

\begin{prop}\label{cty1}
Let $Z\in {\mathcal D}$  be the random function given by Theorem \ref{finlim2} and let $t\in [t_0,t_1)$.  Suppose that for some $\beta>0$,
\be\label{Obound}
\sum_{(x,y)\in \Pi^+\, :\,  t< x\leq t+h}  y^{-2/\alpha(Z(t))} \ = \ O(h^\beta) \qquad (0<h<t_1 - t).
\ee 
Then, conditional  on $\mathcal{F}_{t}$, given $0<\epsilon<\beta$  there exist almost surely random numbers $C_1, C_2 <\infty$  such that  for all $0\leq h < t_1-t$,
\be\label{holas}
|Z(t+h) -Z(t)| \ \leq \ C_1h^{(\beta-\epsilon)/2}.
\ee
If, in addition to  \eqref{Obound},
$$
\sum_{(x,y)\in \Pi^+\, :\,  t< x\leq t+h}  y^{-2/(a,b)} \ = \ O(h^\gamma) \qquad (0<h< t_1 - t).
$$ 
then
\be\label{locas}
\big|\big(Z(t+h)-Z(t)\big) -L(t+h)\big|\  \leq \ C_2 h^{(\beta+\gamma-\epsilon)/2}\qquad (0<h< t_1 - t),
\ee
where $L$ is as in  \eqref{limL} and defined using the same realisation of $\{S(x,y): (x,y) \in \Pi^+, t\leq x< t_1\}$ as $Z$.
\end{prop}

\begin{proof}
Setting $h = 2^{-k}(t_1 - t)$ in  \eqref{espest2}, multiplying by  $2^{k(\beta- \epsilon)}$ and summing, 
$$
\E\Big(\sum_{k=0}^\infty 2^{k(\beta- \epsilon)}\sup_{0\leq h'\leq 2^{-k}(t_1-t)}\big(Z(t+h')-Z(t)\big)^2\Big) \ \leq \ 
c\sum_{k=0}^\infty 2^{k(\beta- \epsilon)} 2^{-k\beta} \ < \ \infty
$$
for some constant $c$, giving \eqref{holas}. The bound \eqref{locas} follows in a similar manner using \eqref{espest2L}.
\end{proof}

\section{General Poisson point sums}\label{sec:rand}
\setcounter{equation}{0}
\setcounter{theo}{0}

We apply the conclusions of Section \ref{sec:signs} to random functions where $\Pi_2^+$ is a realisation of  a Poisson point processes in the half-plane and show that this gives a self-stabilizing processes.
The key idea is that the distribution of the point sets $\{(\X,\Y) \in \Pi\}$ where $\Pi \subset (t_0,t_1)\times \mathbb{R}$ is  a Poisson point process with plane Lebesgue measure ${\mathcal L}^2$ as mean measure, is identical to that of 
$\{(\X,S(\X,\Y)\Y): (\X,\Y) \in \Pi_2^+, S(\X,\Y) = \pm 1\}$, where $\Pi_2^+$ is a Poisson point process on $(t_0,t_1)\times \mathbb{R}^+$  with double Lebesgue measure $2{\mathcal L}^2$ as mean measure and with the $S(\X,\Y)$ independently taking the values  $\pm1$ with equal probability $\frac{1}{2}$ for each $(\X,\Y) \in \Pi^+_2$; this follows from the superposition property of Poisson processes, see \cite[Sections 2.2, 5.1]{King}. 
Hence $\Pi$ can be realised by first sampling $(\X,\Y)$ from $\Pi_2^+$ and then assigning random signs to the $\Y$ coordinates. 

\subsection{Existence of random functions}

Given a Poisson process $\Pi \subset (t_0,t_1)\times \mathbb{R}$ and $\alpha :\mathbb{R} \to [a,b]$   with  $0<a <b<2 $, we wish to show that there exist random functions ${\mathcal D}$ satisfying 
\begin{eqnarray}\label{zedinfty}
Z(t)\ =\  a_0 + \sum_{(\X,\Y) \in \Pi } 1_{(t_0,t]}(\X)\, \Y^{\langle-1/\alpha (Z(\X_-))\rangle} \qquad (t_0\leq t <t_1)
\end{eqnarray}
in an appropriate sense. If $0<a< b<1$ then almost surely $\sum_{(\X,\Y) \in \Pi }  |\Y|^{-1/\alpha (Z(\X_-))}$ converges, in which case the sum in \eqref{zedinfty} is almost surely absolutely convergent, but if $\alpha(z)\geq 1$ for some $z$ there is no a priori guarantee of convergence. In a similar way to Section \ref{sec:signs} we define  $Z_n\in {\mathcal D}$ for $n\in \mathbb{N}$ by 
\begin{eqnarray}\label{zedn}
Z_n (t)\ =\  a_0 + \sum_{(\X,\Y) \in \Pi : |\Y|\leq n} 1_{(t_0,t]}(\X)\, \Y^{\langle-1/\alpha (Z_n(\X_-))\rangle} \qquad (t_0\leq t <t_1).
\end{eqnarray}
Almost surely this  sum is over a finite number of points and therefore, conditional on $\Pi$, $Z_n$ is a well-defined piecewise-constant random function. We are interested in convergence of $Z_n$  to a limiting function $Z$ that satisfies \eqref{zedinfty} in some sense. 

The following result, which is part of Campbell's theorem, will be useful in bounding Poisson sums.

\begin{theo}[Campbell's theorem] \label{camp}
Let $\Pi$ be a Poisson process on $S\subset \mathbb{R}^n$ with mean measure $\mu$ and let $f:S \to \mathbb{R}$ be measurable. Then
$$\E\bigg(\sum_{\Pt \in \Pi}f(\Pt)\bigg) = \int_S f(u)d\mu(u)$$
and 
$$\E\bigg(\exp\sum_{\Pt \in \Pi}f(\Pt)\bigg) = \exp\int_S \big( \exp f(u)-1\big)d\mu(u),$$
provided these integrals converge.
\end{theo}
\begin{proof}
See \cite[Section 3.2]{King}.
\end{proof}

\begin{lem} \label{sumest}
Let $\Pi \subset (t_0,t_1)\times \mathbb{R}$ be a Poisson process  with mean measure ${\mathcal L}^2$ and let $\alpha :\mathbb{R} \to [a,b]$ where  $0<a <b<2 $.  Then for all $0< \eta <2/b-1$ there is almost surely a random $C<\infty$ such that  

\be\label{tailsum}
\sum_{(\X,\Y) \in \Pi: |\Y|>n}  |\Y|^{-2/b}\ <\ C\, n^{-\eta}\qquad (n \in \mathbb{N}).
\ee
\end{lem}
\begin{proof}
By Theorem \ref{camp},
$$
\E\Big(\sum_{(\X,\Y) \in \Pi: |\Y|>n}  |\Y|^{-2/b}\Big)\ = \ \int_n^\infty y^{-2/b}
\ \leq \ \frac{b}{2-b}n^{1-2/b}.$$
A Borel-Cantelli argument summing over $n=2^{-k}$ completes the proof.
\end{proof}

We can now obtain develop Theorem \ref{finlim2} to sums over a Poisson point process.

\begin{theo}\label{thmrand2}
Let $\Pi \subset (t_0,t_1)\times \mathbb{R}$ 
be a  Poisson point process with mean measure ${\mathcal L}^2$, let $\alpha :\mathbb{R} \to [a,b]$   where  $0<a <b<2 $ and let $a_0 \in \mathbb{R}$. 
Then there exists $Z \in {\mathcal D}$ satisfying  
\eqref{zedinfty} in the sense that   
 $\lim_{n\to \infty}\E\big(\|Z_n -Z\|_\infty^2\big)=0$ where $Z_n$ is as in  \eqref{zedn}. Moreover, there exists a sequence $n_j \nearrow \infty$ such that almost surely $\|Z_{n_j} -Z\|_\infty \to 0$.  If $0<b<1$ then almost surely $\|Z_{n} -Z\|_\infty \to 0$. 
\end{theo}

\begin{proof}
Since $0< b<2$,  the Poisson point set  $\Pi$ is almost surely a countable set of isolated points with 
$$
 \sum_{(\X,\Y) \in \Pi} 1_{(t_0,t_1]}(\X)|\Y|^{-2/(a,b)}< \infty
$$
and with the $\X$ distinct.
As noted, $Z_n$ in \eqref{zedn} has the same distribution as 
$$
Z_n(t) = a_0+\!\!\sum_{(\X,\Y) \in \Pi^+_2 : |\Y|\leq n} 1_{(t_0,t]}(\X) S(\X,\Y)\Y^{-1/\alpha (Z(\X_-))}, 
$$
where $\Pi_2^+$ is a Poisson point process on $(t_0,t_1)\times \mathbb{R}^+$  with  $2{\mathcal L}^2$ as mean measure and $S(\X,\Y)$ are random signs on $\Pi_2^+$. Thus, by Theorem \ref{finlim2}, for almost all realisations  of   $\Pi_2^+$   there almost surely exists a random function  $Z \in {\mathcal D}$, such that  $\E\big(\|Z_n-Z\|_\infty^2) \to 0$, and also such that there exists a sequence $n_j \nearrow \infty$ with $\|Z_{n_j} -Z\|_\infty \to 0$; note that by \eqref{tailsum} and \eqref{rapcon} we can take the same sequence $n_j$ for all such realisations.
Thus the conclusion holds for almost all sign combinations for almost all realisations of $\Pi_2^+$ and so  for almost all $\{(\X,S(\X,\Y)\Y): (\X,\Y) \in \Pi_2^+, S(\X,\Y) =\pm 1\}$, that is for almost all $\Pi$.
\end{proof}

For purposes of simulating these random functions we would like an estimate on how rapidly  $Z_n$ given by \eqref{zedn} converges to $Z$. However we  cannot get useful estimates directly from  Theorem \ref{finlim2} since allowing $(x,y) \in \Pi^+$ with $y$ arbitrarily small leaves the sum in  \eqref{diffbound}  unbounded. However, we can get some concrete  estimates if we modify the setting slightly by assuming that there is  $y_0>0$ such that  $|\Y|\geq y_0$ if $(\X,\Y) \in \Pi$, which ensures that the right hand side of \eqref{diffbound1} below converges. In practice this is a realistic assumption in that it  excludes the possibility of $Z$ having unboundedly large jumps.

\begin{theo}\label{thmrand3}
Let  $y_0>0$ and  let $\Pi $ be a Poisson point process on $ (t_0,t_1)\times (-\infty, -y_0] \cup [y_0,\infty)$  
 with  mean measure ${\mathcal L}^2$ restricted to this domain. Let $a_0\in\mathbb{R}$, let $0<a <b<2 $, let $\alpha :\mathbb{R} \to [a,b]$ and  let $Z \in {\mathcal D}$ be the random function given by Theorem \ref{thmrand2} using  this $\Pi$ with $Z_n$ as in \eqref{zedn}. Then as $n\to \infty$
\be
\E\big(\|Z_n -Z\|_\infty^2\big)\ \leq \ 
\frac{8b(t_1-t_0)}{2-b}\exp\bigg(2M^2(t_1-t_0)\int_{y_0}^\infty  y^{-2/(a,b)}\ dy\bigg) 
 n^{-(2-b)/b}\ \to 0. 
\label{diffbound1}
\ee

\end{theo}

\begin{proof}
We proceed exactly as in the proofs of Proposition \ref{finlim1A} and Theorem \ref{finlim2}, except that we condition on  realisations  of a Poisson process $\Pi_2^+$ with mean measure $2{\mathcal L}$ on $(t_0,t_1)\times [y_0,\infty)$, to get  \eqref{diffbound} in this setting and then use Theorem \ref{thmrand2} to get the process on this $\Pi$. 
Then for $n\geq y_0$, using \eqref{diffbound}, the independence of the  Poisson point process $\Pi_2^+$ on  $(t_0,t_1)\times [y_0,n]$ and on $(t_0,t_1)\times (n,\infty)$, and Theorem \ref{camp},	
\begin{align*}
\E\big(\|Z_n& -Z\|_\infty^2\big) \ =\ \E\Big(\E\big(\|Z_n -Z\|_\infty^2\big| \Pi_2^+ \big)\Big) \\
&\leq\ 
4\,\E\bigg( \prod_{(\X,\Y) \in \Pi ^+_2: y_0\leq \Y\leq n } (1+M^2\Y^{-2/(a,b)})\sum_{(\X,\Y)\in \Pi ^+_2 : \Y> n} |\Y|^{-2/b}\bigg)\\
&=\ 
4\,\E\bigg( \exp\Big( \sum_{(\X,\Y) \in \Pi ^+_2: y_0\leq \Y\leq n } \log(1+M^2\Y^{-2/(a,b)})\Big)\bigg)
\E\bigg( \sum_{(\X,\Y)\in \Pi ^+_2 : \Y> n} |\Y|^{-2/b}\bigg)\\
&=\ 
4\,\exp\bigg(\int_{y_0}^n 2(t_1-t_0)\ M^2 y^{-2/(a,b)} dy\bigg) 
(t_1-t_0) \int_n^\infty  2y^{-2/b}  dy.
 \end{align*}
Letting $n\to \infty$ in the first integral and evaluating the second integral gives \eqref{diffbound1}.
\end{proof}

We remark that Theorem \ref{thmrand3} allows us to quantify the rate of convergence in probability of $\|Z_n -Z\|_\infty\to 0$ in Theorem \ref{thmrand2}. By the Poisson distribution $\P\{\Pi \cap ((t_0,t_1)\times [0,y_0] )= \emptyset\} = \exp(-y_0(t_1-t_0))$. Given $\epsilon >0$ we can choose $y_0$ to make this probability at most $\epsilon/2$, then using  \eqref{diffbound1} and Markov's inequality it follows that if  $n$ is sufficiently large then   $\P\{ \|Z_n -Z\|_\infty>\epsilon\} < \epsilon$. In practice, this leads to an enormous value of $n$. 

\subsection{Local properties and self-stabilising processes}

We next obtain local properties of the random functions defined by a Poisson point process as in Theorem \ref{thmrand2}. Not only are the sample paths   right-continuous, but they satisfy a local H\"{o}lder continuity estimate and are self-stabilizing, that is locally they look like $\alpha$-stable processes. 

We will use a bound provided by the $\alpha$-{\it stable subordinator} which may be defined for each (constant) $0<\alpha<1$ by
$$
S_\alpha (t) := \sum_{(\X,\Y) \in \Pi} 1_{(t_0 ,t]}(\X)\, |\Y|^{-1/\alpha}\qquad (t_0 \leq t <t_1),
$$
where the sum, which is almost surely convergent, is over a plane Poisson point process $\Pi$ with mean measure ${\mathcal L}^2$.
Then $S_\alpha $ on $[t_0 ,t_1)$ has stationary increments and for each $0<\epsilon<1/\alpha$ satisfies the H\"{o}lder property
\be\label{sub2}
S_\alpha (t)\   \leq\ C (t-t_0)^{(1/\alpha)\,-\,\epsilon},
\ee
where $C$ is almost surely finite; this may be established using Campbell's Theorem \ref{camp} in a similar way to the proof of Lemma \ref{sumest}, or see  \cite[Section III.4]{Ber} or \cite{Tak}.

We write $L^0_\alpha $ for the non-normalized $\alpha$-stable process, which has a representation
\begin{equation}\label{sumnn}
L^0_\alpha  (t) =\sum_{(\X,\Y) \in \Pi} 1_{(0,t]}(\X) \Y^{<-1/\alpha>}
\end{equation}
where $\Pi$ is a Poisson point process on $(t_1,t_2) \times \mathbb{R}$ with mean measure ${\mathcal L}^2$.
This sum is almost surely absolutely convergent if $0<\alpha<1$ but for general $0<\alpha<2$
it is the limit as $n \to \infty$ of
$$L^0_{\alpha,n}(t) =\sum_{(\X,\Y) \in \Pi:|\Y| \leq n} 1_{(0,t]}(\X) \Y^{<-1/\alpha>}.$$ 
Whilst $\E\big(\|L^0_{\alpha,n}  -L^0_\alpha\|_\infty^2\big)\to 0$ as a special case of Theorem \ref{thmrand2}, the constant value of $\alpha$ means that $L^0_{\alpha,n}$ and $L^0_{\alpha,m}-L^0_{\alpha,n}$ are independent for $m>n$ and also that $\{L^0_{\alpha,n}\}_n$ is a martingale, which ensures that $\|L^0_{\alpha,n}  -L^0_\alpha\|_\infty\to 0$ almost surely.  

In the same way to $Z_n$ we can think of $L^0_{\alpha,n}(t)$ in terms of  $\Pi^+_2$ and random signs, so that 
$$L^0_{\alpha,n}(t) =\sum_{(\X,\Y) \in \Pi^+_2:\Y \leq n} 1_{(0,t]}(\X) S(\X,\Y)\Y^{<-1/\alpha>}.$$

The following proposition is the analogue of Proposition \ref{cty1} in this context.

\begin{prop}\label{cty3}
Let $Z$  be the random function given by Theorem \ref{thmrand2} and let $t\in [t_0,t_1)$. Then, conditional on ${\mathcal F}_{t}$, given $0< \epsilon< 1/b$ there exist almost surely random numbers $C_1, C_2 <\infty$  such that  for all $0\leq h <t_1-t$,
\be\label{holas3}
|Z(t+h) -Z(t)|  \leq C_1 h^{1/\alpha(Z(t))\,-\,\epsilon}.
\ee
and
\be\label{locas3}
\big|\big(Z(t+h)-Z(t)\big)   -\big(L^0_{\alpha(t)}(t+h)-L^0_{\alpha(t)}(t)\big)\big|
 \ \leq \ C_2 h^{1/\alpha(Z(t))+ 1/b -\,\epsilon},
\ee
where  $L^0_{\alpha(t)}$ is the $\alpha(t)$-stable process \eqref{sumnn} defined using the same realisations of $\Pi$ as $Z$.
\end{prop}

\begin{proof} 
Let $\epsilon>0$. Let  $\Pi^+_2$ be a Poisson point process on $\mathbb{R}$ with mean measure $2{\mathcal L}$.
From \eqref{sub2} 
$$
\sum_{(X,Y)\in \Pi^+_2\, :\,  t< x\leq t+h}  \Y^{-2/\alpha(Z(t))}
\ =\ 2S_{\alpha(Z(t))/2}(h)\ \leq\ C h^{2/\alpha(Z(t)) -\epsilon}
$$
where $C<\infty$ for almost all realisations of $\Pi^+_2$.
For such $\Pi^+_2$,  Proposition \ref{cty1} gives on randomising the signs,
$$\big|Z(t+h) -Z(t)|\Pi^+_2\big|\  \leq\ C_1 h^{1/\alpha(Z(t))-\epsilon/2-\epsilon/2}$$
for some random $C_1$, 
and hence \eqref{holas3} holds almost surely.

In the same way,
$$\sum_{(X,Y)\in \Pi^+_2\, :\,  t< x\leq t+h}  \Y^{-2/(a,b)}\ \leq\ C' (h^{-1})^{-2/(a,b) -\epsilon/2} \ \leq\ C^{\prime\prime}h^{2/b -\epsilon} 
$$
where $C', C^{\prime\prime}< \infty$ for almost all realisations of $\Pi^+_2$.
Then in Proposition \ref{cty1}, $L(t+h) = L^0_{\alpha(t)}(t+h)-L^0_{\alpha(t)}(t)$, so  for such $\Pi^+_2$, on randomising the signs,
$$\Big|\big(Z(t+h)-Z(t)\big)   -\big(L^0_{\alpha(t)}(t+h)-L^0_{\alpha(t)}(t)\big)\big|\Pi^+_2\Big| \  \leq\ C_2 h^{1/\alpha(Z(t))+ 1/b -\,\epsilon}$$
for random $C_2<\infty$, so \eqref{locas3} holds almost surely.
\end{proof}

We finally show that almost surely at each $t\in [t_0,t_1)$ the random function $Z$ of Theorem \ref{thmrand2} is right-localisable with local form  an $\alpha(Z(t))$-stable process, so that $Z$ may indeed be thought of as self-stablizing. 

\begin{theo}\label{rtloc}
Let $Z$ be the random function given by Theorem \ref{thmrand2} and let $t\in [t_0,t_1)$. Then, conditional on ${\mathcal F}_{t}$, almost surely $Z$ is strongly right-localisable at $t$, in the sense that
$$\frac{Z(t+ru) -Z(t)}{r^{1/\alpha(Z(t))}}\bigg|\,  \mathcal{F}_t \  \tod \ L^0_{\alpha(Z(t))}(u) \qquad (0\leq u \leq 1)
$$
as $r\searrow 0$, where convergence is in distribution with respect to $(D[0,1], \rho_S)$, with $ \rho_S$ is the Skorohod metric.
\end{theo}

\begin{proof}
Let $0<\epsilon< 1/b$. For $u\in [0,1]$   and $0<r<t_1-t$, almost surely 
\begin{align*}
\big|\big(Z(t+  ru) & -Z(t)\big) -\big(L^0_{\alpha(t)} (t+ru)-L^0_{\alpha(t)} (t)\big)\big| \\
& \leq\  C_2 (ru)^{1/\alpha(Z(t))+ 1/b -\,\epsilon}
\ \leq\ C_2 r^{1/\alpha(Z(t))+ 1/b -\,\epsilon}
\end{align*}
for a random $C_2<\infty$, by Proposition \ref{cty3}. Thus
\begin{eqnarray*}
\bigg\| \frac{Z(t+  ru) - Z(t)}{r^{1/\alpha(Z(t))}}\  -\ \frac{L^0_{\alpha(Z (t))}(t+ru)-L^0_{\alpha(Z (t))}(t)}{ r^{1/\alpha(Z(t))}}\bigg\|_\infty 
&\leq&
C_2 \frac{r^{1/\alpha(Z(t))+ 1/b -\,\epsilon}}{ r^{1/\alpha(Z(t))}}\\
&=& C_2  r^{1/b-\epsilon} \to 0
\end{eqnarray*}
almost surely as $r\searrow 0$. In particular, since $\|\cdot\|_\infty$ dominates $\rho_S$ on $D[0,1]$,
$$ \rho_S \bigg( \frac{Z(t+  ru) - Z(t)}{r^{1/\alpha(Z(t))}}, \frac{L^0_{\alpha(Z (t))}(t+ru)-L^0_{\alpha(Z (t))}(t)}{ r^{1/\alpha(Z(t))}}    \bigg) \topp 0$$
almost surely and in probability. Using that $\alpha$-stable processes have stationary increments and scale with exponent $1/\alpha$,
$$\frac{L^0_{\alpha(Z (t))}(t+ru)-L^0_{\alpha(Z (t))}(t)}{ r^{1/\alpha(Z(t))}} 
\ \ed \  L^0_{\alpha(Z (t))}(u)-L^0_{\alpha(Z (t))}(0) \ed L^0_{\alpha(Z (t))}(u),$$
so we conclude, using \cite[Theorem 4.1]{Bil} to combine convergence in probability and in distribution, that  
$$\frac{Z(t+  ru) - Z(t)}{r^{1/\alpha(Z(t))}}\bigg|\,  \mathcal{F}_t \  \tod \ L^0_{\alpha(Z (t))}(u)$$
 as $r\searrow 0$.
\end{proof}

\section*{Acknowledgements}
The authors thank the referee for some helpful comments. KJF gratefully acknowledges the hospitality of Institut Mittag-Leffler in  Sweden, where part of this work was carried out. JLV is grateful to SMABTP for financial support.

\bibliographystyle{plain}

\bigskip
\end{document}